\documentclass{amsart}
\usepackage{amssymb}
\usepackage[top= 1.5in, bottom=1.5in, left=1.5in, right=1.5in]{geometry}
\usepackage{mathrsfs}
\usepackage{graphicx}
\usepackage{comment}
\usepackage{float}  
\allowdisplaybreaks

\newcommand{\be}{\begin{equation}}
\newcommand{\ee}{\end{equation}}
\newcommand{\ba}{\begin{align}}
\newcommand{\ea}{\end{align}}

\newtheorem{theorem}{Theorem}[section]
\newtheorem{lemma}[theorem]{Lemma}

\newtheorem*{proposition*}{Proposition}
\newtheorem{corollary}[theorem]{Corollary}
\newtheorem*{theorem*}{Theorem}
\newtheorem*{corollary*}{Corollary}
\newtheorem*{cor*}{Corollary}

\theoremstyle{definition}
\newtheorem{definition}[theorem]{Definition}

\newtheorem{remark}[theorem]{Remark}

{\kern -4pt}

{\begin{list}{}{%
\settowidth{\labelwidth}{\textsf{{\it #1.}}}%
\setlength{\labelsep}{4mm}%
\setlength{\leftmargin}{\labelwidth}%
\addtolength{\leftmargin}{\labelsep}%
}}%
{\end{list}}

\def\beq{\begin{equation}}\def\enq{\end{equation}}

{\begin{list}{}{%
\settowidth{\labelwidth}{\textsf{{\it #1.}}}%
\setlength{\labelsep}{2mm}%
\setlength{\leftmargin}{\labelwidth}%
\addtolength{\leftmargin}{\labelsep}%
\addtolength{\leftmargin}{4mm}%
\setlength{\itemsep}{6pt}%
\setlength{\listparindent}{0pt}%
\setlength{\topsep}{3pt}%
}}%

\usepackage{color}
\usepackage[normalem]{ulem}
\usepackage{soul}

\usepackage{xcolor}
\newcommand{\smat}[4]{\left(\begin{smallmatrix}#1&#2\\#3&#4\end{smallmatrix}\right)}
\DeclareMathOperator{\ord}{ord}
\DeclareMathOperator{\Ord}{Ord}

\begin{document}

\title[Rogers--Ramanujan Continued Fraction]
{Eta-Quotient Representations for a {Three Parameter} Family of Modular Functions Associated with the Rogers--Ramanujan Continued Fraction}

\author[B. Paudel]{Bishnu Paudel}

\author[J. Sellers]{James A. Sellers}

\author[H. Wang]{Haiyang Wang}

\address{Mathematics and Statistics Department\\
University of Minnesota Duluth\\
Duluth, MN 55812, USA}
\email{bpaudel@d.umn.edu, jsellers@d.umn.edu, wan02600@d.umn.edu}

\begin{abstract}
In 2021, Chern and Tang introduced two families of two-parameter modular
functions associated with the Rogers--Ramanujan continued fraction.  They established recurrence relations that express the
members of these families in terms of eta quotients, and used these expressions to
obtain dissection formulas.
Motivated by their work, we construct a three-parameter
extension and derive the corresponding recursive eta quotient
representations. Our result also has applications to dissection formulas. In particular,
it is used in separate work to obtain $5$-dissection formulas for overpartitions with restricted odd differences.
\end{abstract}

\maketitle

\section{Introduction}
Let $\mathbb{H}$ denote the upper half-plane of the complex plane. For $\tau\in\mathbb{H}$ set
$q:=e^{2\pi i\tau}$, so that $|q|<1$. {Throughout this work, we} use the {standard Pochhammer} notation
\begin{equation*}
(a;q)_\infty:=\prod_{k=0}^{\infty}(1-aq^k),
\qquad
f_j:=(q^j;q^j)_\infty
\end{equation*}
{where $j\in\mathbb{Z}_{\ge 1}$.}
The Dedekind eta function is {given by}
\begin{equation*}
\eta(\tau)=\eta(q):=q^{1/24}(q;q)_\infty.
\end{equation*}
The Rogers--Ramanujan continued fraction is defined by
\begin{equation*}
r(\tau)=r(q):=
\cfrac{q^{1/5}}
{1+\cfrac{q}
{1+\cfrac{q^2}
{1+\cfrac{q^3}
{1+\cdots}}}}
=
q^{1/5}
\frac{(q;q^5)_\infty(q^4;q^5)_\infty}
{(q^2;q^5)_\infty(q^3;q^5)_\infty}.
\end{equation*}
For the product representation, see \cite[Theorem 7.3.3]{Berndt06}. This
continued fraction first appeared in the work of Rogers \cite{Rogers} in 1893, and
was rediscovered independently by Ramanujan, who recorded it in his first
letter to Hardy, dated January~16, 1913, and by Schur \cite{Schur}.
It is convenient to remove the
fractional power of $q$ and set
\begin{equation*}
R(q):=q^{-1/5}r(q).
\end{equation*}
For each positive integer $j$, we also write
\begin{equation*}
r_j(\tau):=r(j\tau)=q^{j/5}R(q^j).
\end{equation*}

Ramanujan recorded many remarkable identities satisfied by $R(q)$. Among
them are the reciprocal identities
\begin{align*}
\frac{1}{q^{1/5}R(q)}-1-q^{1/5}R(q)
&=\frac{(q^{1/5};q^{1/5})_\infty}{q^{1/5}f_5},\\
\frac{1}{qR(q)^5}-11-qR(q)^5
&=\frac{f_1^6}{qf_5^6}.
\end{align*}
See \cite[Sec.~7.4]{Berndt06}. Many later works study identities of this type. See, {for example},
\cite{Aygin,Baruah-Begum,CT2021,Gugg2012,HS}. In particular, for
$m\in\mathbb{Z}_{\geq 0}$ and $n\in\mathbb{Z}$, Chern and Tang
\cite{CT2021} introduced the two families
\begin{align*}
P(m,n)
&:=
\frac{1}{q^mR(q)^{m+2n}R(q^2)^{2m-n}}
+(-1)^{m+n}q^mR(q)^{m+2n}R(q^2)^{2m-n},\\
Q(m,n)
&:=
\frac{1}{q^mR(q)^{2m+3n}R(q^3)^{m-n}}
+(-1)^mq^mR(q)^{2m+3n}R(q^3)^{m-n},
\end{align*}
and used recurrence relations to express them in terms of the eta quotients
\begin{align*}
K&:=\frac{\eta(q^2)\eta(q^5)^5}
{\eta(q)\eta(q^{10})^5}
=q^{-1}\frac{f_2f_5^5}{f_1f_{10}^5},
\\
S&:=\frac{\eta(q)^3\eta(q^3)^3}
{\eta(q^5)^3\eta(q^{15})^3}
=q^{-2}\frac{f_1^3f_3^3}{f_5^3f_{15}^3},\\
T&:=\frac{\eta(q^3)\eta(q^5)^5}
{\eta(q)\eta(q^{15})^5}
=q^{-2}\frac{f_3f_5^5}{f_1f_{15}^5}.
\end{align*}
They also applied their results to derive dissection formulas for certain
partition functions. For further applications of their results, see~\cite{CLY,Gua2025,Tang25}.

In this paper we extend the construction of Chern and Tang to a
three-parameter family. For
$(m,n,p)\in\mathbb{Z}^3$, set
\begin{equation}
w=w(m,n,p):=
q^{-p}R(q)^{-2m-n}R(q^2)^{m-n-p}R(q^3)^{n-p}
=
r_1^{-2m-n}r_2^{m-n-p}r_3^{n-p},
\label{eq.w-def}
\end{equation}
and define
\begin{equation}
G(m,n,p):=w+\frac{(-1)^{m+n}}{w}.
\label{eq.G-def}
\end{equation}

Directly from \eqref{eq.w-def} and \eqref{eq.G-def}, {we see that, for any $m, n, p$,}
\begin{equation*}
G(-m,-n,-p)=(-1)^{m+n}\,G(m,n,p),
\end{equation*}
and the families of Chern and Tang are recovered from $G$ by
\begin{equation*}
P(m,n)=G(n,m,m),
\qquad
Q(m,n)=G(m+n,n,m).
\end{equation*}
We also require the eta quotients
\begin{align}
U&:=\frac{\eta(q^2)^3\eta(q^3)^3}
{\eta(q)^2\eta(q^6)^2\eta(q^{10})\eta(q^{15})}
=q^{-1}\frac{f_2^3f_3^3}{f_1^2f_6^2f_{10}f_{15}},
\label{eq.Udef}\\
V&:=\frac{\eta(q)^2\eta(q^6)\eta(q^{30})}
{\eta(q^3)\eta(q^{10})^2\eta(q^{15})}
=\frac{f_1^2f_6f_{30}}{f_3f_{10}^2f_{15}}.\label{eq.Vdef}
\end{align}

{The primary goal in this work is to extend the recurrence results of Chern and Tang to this 3--parameter setting.  Our}
main result is as follows.

\begin{theorem}\label{thm.main}
For all integers $m,n,p$, we have the recurrence relations:
\begin{align}
G(m+1,n,p) &= 4K^{-1}\,G(m,n,p) + G(m-1,n,p), \label{eq.recm}\\
G(m,n+1,p) &= (1-V)\,G(m,n,p) + G(m,n-1,p), \label{eq.recn}\\
G(m,n,p+1) &= (U-2)\,G(m,n,p) - G(m,n,p-1). \label{eq.recp}
\end{align}
Moreover, the following initial values hold:
\begin{align}
G(0,0,0) &= 2,\label{eq.000}\\
G(1,0,0) &= 4K^{-1},\label{eq.100}\\
G(0,1,1) &= K,\label{eq.011}\\
G(1,1,1) &= K + 2 + 4K^{-1},\label{eq.111}\\
G(1,1,0) &= 2 + 9T^{-1},\label{eq.110}\\
G(1,0,1) &= \tfrac14\bigl(T - S + 9T^{-1} + 6\bigr), \label{eq.101}\\
G(0,1,0) &= 1 - V,\label{eq.010}\\
G(0,0,1) &= U - 2.\label{eq.001}
\end{align}
\end{theorem}

The identities \eqref{eq.100} and \eqref{eq.011} are, respectively,
(1.20) and (1.19) of Baruah and Begum
\cite[Lemma~1.3]{Baruah-Begum}. Gugg proved \eqref{eq.110} in
\cite[Theorem~5.1(iv)]{Gugg2012}. Identities \eqref{eq.111} and \eqref{eq.101} appear in
Chern and Tang \cite[Table~1 and (1.14)]{CT2021}. The identities~\eqref{eq.010} and
\eqref{eq.001} are proved in Section~\ref{sec.init}. The proofs of the recurrences \eqref{eq.recm}--\eqref{eq.recp}  are given in
Section~\ref{sec.product}.

\begin{remark}
Suppose that $(a,b,c)\in\mathbb{Z}^{3}$ satisfies
$a+2b+3c\equiv0\pmod5$, and define
\begin{equation*}
t=\tfrac25(a+2b+3c).
\end{equation*}
Then, using \eqref{eq.w-def} and \eqref{eq.G-def}, we have
\begin{equation*}
\begin{aligned}
&R(q)^{a}R(q^{2})^{b}R(q^{3})^{c}
+
\frac{(-1)^{a+b+c}}
{q^{t}R(q)^{a}R(q^{2})^{b}R(q^{3})^{c}}\\
&\qquad
=q^{-t/2}G\bigl(
\tfrac15(-2a+b-c),\;
\tfrac15(-a-2b+2c),\;
-\tfrac15(a+2b+3c)
\bigr),
\end{aligned}
\end{equation*}
which is useful for rewriting expressions of the type on the left-hand side in
terms of $G(m,n,p)$.
\end{remark}

Theorem~\ref{thm.main} is one of the main tools {utilized} in \cite{PSW}, where it is used
to obtain $5$-dissection formulas and congruences modulo $5$ for {the} function {$\overline{t}(n)$ which counts the number of overpartitions of $n$ with restricted odd differences. For further details on {$\overline{t}(n)$, see \cite{BDLM, HS2023}. 

The recurrence relations and initial values in
Theorem~\ref{thm.main} imply the following.
\begin{corollary}
For every $(m,n,p)\in\mathbb{Z}^{3}$,
\begin{equation*}
G(m,n,p)\in
\mathbb{Q}\bigl[K^{\pm1},\,S,\,T^{\pm1},\,U,\,V\bigr].
\end{equation*}
\end{corollary}

{The above is an extension of the work on the corresponding 2--parameter result of Chern and Tang \cite[Corollaries 1.5 and 1.6]{CT2021}.}

The remainder of the paper is organized as follows.
Section~\ref{sec.prelim} reviews the {necessary} background on modular functions. The initial values \eqref{eq.010} and \eqref{eq.001} are
established in Section~\ref{sec.init}. Section~\ref{sec.product} derives the
recurrences \eqref{eq.recm}, \eqref{eq.recn}, and \eqref{eq.recp},  which, combined with these initial values, complete the proof of
Theorem~\ref{thm.main}. 

\section{Preliminaries}\label{sec.prelim}
We first recall some background on modular functions. For
further details, see \cite{DS,Ono2004}. For a positive integer $N$, the \emph{principal congruence subgroup
of level $N$} is
\begin{equation*}
\Gamma(N):=
\left\{
\smat abcd\in\mathrm{SL}_{2}(\mathbb{Z}) :
a\equiv d\equiv1,\quad b\equiv c\equiv0\pmod N
\right\}.
\end{equation*}
A subgroup $\Gamma\leq\mathrm{SL}_{2}(\mathbb{Z})$ is called a
\emph{congruence subgroup} if it contains $\Gamma(N)$ for some
positive integer $N$. Our primary interest is the congruence subgroup
\begin{equation*}
\Gamma_{0}(N):=
\left\{
\smat abcd\in\mathrm{SL}_{2}(\mathbb{Z}) :
c\equiv0\pmod N
\right\}, 
\end{equation*}
which contains $\Gamma(N)$. The group $\mathrm{SL}_{2}(\mathbb{Z})$, and hence each of its
congruence subgroups, acts on the extended upper half-plane
\begin{equation*}
\mathbb{H}^{*}:=\mathbb{H}\cup\mathbb{Q}\cup\{\infty\}
\end{equation*}
by fractional linear transformations
\begin{equation*}
\gamma\tau=\frac{a\tau+b}{c\tau+d},
\qquad
\gamma=\smat abcd\in\mathrm{SL}_{2}(\mathbb{Z}).
\end{equation*}

The \emph{cusps} of a congruence subgroup $\Gamma$ are the
equivalence classes of $\mathbb{Q}\cup\{\infty\}$ under this action.
Let $s$ be a cusp of $\Gamma$, and choose
$\gamma\in\mathrm{SL}_{2}(\mathbb{Z})$ such that
$\gamma(\infty)=s$. The \emph{width} of $\Gamma$ at $s$, denoted
$h_{s}=h(\Gamma,s)$, is the least positive integer $h$ for which
\begin{equation*}
\gamma\smat{1}{h}{0}{1}\gamma^{-1}\in\pm\Gamma.
\end{equation*}
This definition is independent of the choice of $\gamma$.

\begin{definition}\label{def.modfn}
Let $\Gamma$ be a congruence subgroup. A meromorphic function
$F(\tau)\colon\mathbb{H}\to\mathbb{C}$ is a \emph{modular function}
on $\Gamma$ if
\begin{enumerate}
\item 
$F(\gamma\tau)=F(\tau)$ for all $\gamma\in\Gamma$,
\item for 
every $\gamma\in\mathrm{SL}_{2}(\mathbb{Z})$, there exist $n=n(\gamma)\in\mathbb{Z}_{\ge 1}$ and $m_0=m_0(\gamma)\in\mathbb{Z}$ such that the function $F(\gamma\tau)$ has an
expansion
\begin{equation*}
F(\gamma\tau)
=\sum_{m\geq m_0}c_{\gamma}(m)q_n^m,
\qquad
q_n:=e^{2\pi i\tau/n}.
\end{equation*}
\end{enumerate}
\end{definition}

Let $F(\tau)$ be a modular function on $\Gamma$,
let $s$ be a cusp of $\Gamma$ of width $h$, and let
$\gamma\in\mathrm{SL}_{2}(\mathbb{Z})$ satisfy $\gamma(\infty)=s$.
In Definition~\ref{def.modfn}, one may choose $n(\gamma)=h$.
With this choice, the \emph{order} of $F$
at $s$ is
\begin{equation*}
\Ord_{\Gamma}(F,s)
:=\min\{m\in\mathbb{Z}:c_{\gamma}(m)\neq0\}.
\end{equation*}

Writing the same expansion in terms of $q=e^{2\pi i\tau}$, we call
its least exponent the \emph{invariant order} $\ord(F,s)$, which does not depend on $\Gamma$. Since
$q_h=q^{1/h}$,
\begin{equation}\label{eq.twoorders}
\Ord_{\Gamma}(F,s)
=h(\Gamma,s)\,\ord(F,s).
\end{equation}

We next recall explicit descriptions of the cusps of $\Gamma_{0}(N)$
and their widths.

\begin{lemma}[{\cite[p.~354]{CL2004}}]\label{lem.cusps}
Let $N$ be a positive integer, and set $e_{d}:=\gcd(d,N/d)$ for each
$d\mid N$. For every $d\mid N$, choose a set
\begin{equation*}
S_{d}=\Bigl\{\tfrac{x_{1}}{d},\ \tfrac{x_{2}}{d},\ \dots\Bigr\},
\qquad
\gcd(x_{i},d)=1,\quad 0\le x_{i}\le d-1,\quad
x_{i}\not\equiv x_{j}\pmod{e_{d}}\ \ (i\neq j),
\end{equation*}
maximal with respect to these properties. Then
$\mathcal{S}:=\bigcup_{d\mid N}S_{d}$ is a complete set of
inequivalent cusps of $\Gamma_{0}(N)$.
\end{lemma}

\begin{lemma}[{\cite[Lemma~1.1]{Biagioli1989}}]\label{lem.width}
Let $a,c\in\mathbb{Z}$ with $\gcd(a,c)=1$. Then the width of
$\Gamma_{0}(N)$ at the cusp $a/c$ is
\begin{equation*}
h_{a/c}=\frac{N}{\gcd(N,c^{2})}.
\end{equation*}
\end{lemma}

For a positive integer $N$ and integers $m_{\delta}$
indexed by the divisors $\delta$ of $N$, consider the
eta quotient
\begin{equation}\label{eq.eta}
F(\tau)=\prod_{\delta\mid N}\eta(\delta\tau)^{m_{\delta}}.
\end{equation}
The next two lemmas provide a criterion for such an $F(\tau)$ to be a modular function on $\Gamma_{0}(N)$ and a formula for its orders at the cusps.

\begin{lemma}[{\cite[Theorem 1]{Newman1959}}]
\label{lem.newman}
Let $F(\tau)$ be the eta quotient in \eqref{eq.eta}. Then $F(\tau)$ is a
modular function on $\Gamma_{0}(N)$ if 
\begin{enumerate}
\item $\displaystyle\sum_{\delta\mid N}m_{\delta}=0$,
\item $\displaystyle\sum_{\delta\mid N}\delta m_{\delta}
\equiv0\pmod{24}$,
\item $\displaystyle\sum_{\delta\mid N}\frac{N}{\delta}m_{\delta}
\equiv0\pmod{24}$,
\item $\displaystyle\prod_{\delta\mid N}\delta^{m_{\delta}}$
is the square of a rational number.
\end{enumerate}
\end{lemma}

\begin{lemma}[{\cite[Prop~3.2.8]{Ligozat1975}}]
\label{lem.ligorder}
Let $F(\tau)$ be the eta quotient in \eqref{eq.eta}, and suppose that
$F(\tau)$ is a modular function on $\Gamma_{0}(N)$. Let $a$ and $c$ be
positive integers such that $c\mid N$ and $\gcd(a,c)=1$. Then
\begin{equation*}
\Ord_{\Gamma_{0}(N)}\left(F,\frac{a}{c}\right)
=
\frac{N}{24c\gcd(c,N/c)}
\sum_{\delta\mid N}
\frac{\gcd(c,\delta)^{2}m_{\delta}}{\delta}.
\end{equation*}
\end{lemma}

Let $N\ge 1$ be an integer, and let $g\in\mathbb{Z}$ satisfy
$N\nmid g$. Following \cite{Yang2004} (see also \cite{Gua_RR}), the \emph{generalized Dedekind eta function}
$\eta_{N,g}$ is defined on $\mathbb{H}$ by
\begin{equation}\label{eq.genetadef}
\eta_{N,g}(\tau)
:=q^{NB_{2}(g/N)/2}
\prod_{m=1}^{\infty}
\bigl(1-q^{N(m-1)+g}\bigr)
\bigl(1-q^{Nm-g}\bigr),
\end{equation}
where $q=e^{2\pi i\tau}$ and
\begin{equation*}
B_{2}(x):=x^{2}-x+\tfrac16
\end{equation*}
is the second Bernoulli polynomial. Replacing $\tau$ by $j\tau$
in \eqref{eq.genetadef} gives
\begin{equation}\label{eq.dilation}
\eta_{N,g}(j\tau)=\eta_{jN,jg}(\tau)
\qquad (j\in\mathbb{Z}_{\ge 1}).
\end{equation}

\begin{lemma}[{\cite[Corollary~2]{Yang2004}}]
Let $N\ge1$ and $g\in\mathbb{Z}$ with $N\nmid g$.
\begin{enumerate}
\item[(i)] We have
\begin{equation}\label{eq.Nshift}
\eta_{N,g+N}(\tau)=\eta_{N,-g}(\tau)=-\eta_{N,g}(\tau).
\end{equation}
\item[(ii)] For every $b\in\mathbb{Z}$,
\begin{equation}\label{eq.translation}
\eta_{N,g}(\tau+b)=e^{\pi ibNB_{2}(g/N)}\,\eta_{N,g}(\tau).
\end{equation}
\item[(iii)] If $\gamma=\smat abcd\in\Gamma_{0}(N)$ with $c\ne0$, then
\begin{equation}\label{eq.cne0}
\eta_{N,g}(\gamma\tau)
=\varepsilon(a,bN,c/N,d)\,
e^{\pi i\left(\frac{g^{2}ab}{N}-gb\right)}\,
\eta_{N,ag}(\tau),
\end{equation}
where 
\begin{equation*}
\varepsilon(a,b,c,d)=
\begin{cases}
e^{\pi i\left(\frac{bd(1-c^{2})+c(a+d-3)}{6}\right)}, & c\ \text{odd},\\[3pt]
-i\,e^{\pi i\left(\frac{ac(1-d^{2})+d(b-c+3)}{6}\right)}, & d\ \text{odd}.
\end{cases}
\end{equation*}
\end{enumerate}
\end{lemma}

Equation~\eqref{eq.Nshift} implies that
\begin{equation}
\eta_{N,g+kN}(\tau)=(-1)^k\eta_{N,g}(\tau),
\label{eq.indexshift}
\end{equation}
for every
$k\in\mathbb{Z}$, and
\begin{equation}
\eta_{N,N-g}(\tau)=\eta_{N,g}(\tau).
\label{eq.reflect}
\end{equation}

\begin{lemma}[{\cite[Lemma~2]{Yang2004}}]\label{lem.yangorder}
Let $N\ge1$ be an integer, let $g\in\mathbb{Z}$ satisfy $N\nmid g$, and let
$\gamma=\smat abcd\in\mathrm{SL}_{2}(\mathbb{Z})$.  Then the first term of
the $q$-expansion of $\eta_{N,g}(\gamma\tau)$ is $\varepsilon q^{\delta}$,
where $q=e^{2\pi i\tau}$, $|\varepsilon|=1$, and
\begin{equation*}
\delta=\frac{\gcd(c,N)^{2}}{2N}\,
P_{2}\!\left(\frac{ag}{\gcd(c,N)}\right)
\end{equation*}
with $P_{2}(t):=B_{2}(\{t\})$, where $\{t\}$ is the fractional part of $t$.
\end{lemma}

We will also use the following consequence of the valence formula. See
\cite[Theorem~4.1.4]{Rankin1977}.
\begin{lemma}\label{lem.valence}
Let $\mathcal{C}$ be a complete set of inequivalent cusps of
$\Gamma_{0}(N)$, and let $F$ be a modular function on $\Gamma_{0}(N)$
that is holomorphic on $\mathbb{H}$. Suppose that integers $D_{s}$
satisfy
\begin{equation*}
\Ord_{\Gamma_{0}(N)}(F,s)\ge D_{s}
\qquad\text{for each }s\in\mathcal{C}\setminus\{\infty\},
\end{equation*}
and
\begin{equation*}
\Ord_{\Gamma_{0}(N)}(F,\infty)>
-\sum_{s\in\mathcal{C}\setminus\{\infty\}}D_{s}.
\end{equation*}
Then $F=0$.
\end{lemma}

\section{Proofs of Equations~\eqref{eq.010} and \eqref{eq.001}}\label{sec.init}
In this section we prove the initial values \eqref{eq.010} and
\eqref{eq.001}.  We first establish two preliminary results. As observed in \cite[p.~56]{Gua_RR}, we have
\begin{equation*}
r(\tau)=\frac{\eta_{5,1}(\tau)}{\eta_{5,2}(\tau)},
\end{equation*}
and hence, by \eqref{eq.dilation},
\begin{equation}\label{eq.rgeneta}
r_{j}(\tau)=\frac{\eta_{5j,j}(\tau)}{\eta_{5j,2j}(\tau)}
\qquad(j=1,2,3).
\end{equation}
Throughout this section we write $\zeta:=e^{2\pi i/5}$. The next lemma describes the transformation behavior of $r_{j}(\tau)$, $j\in\{1,2,3\}$, under the action of $\Gamma_{0}(30)$. Observe first that if
$\smat abcd\in\Gamma_{0}(30)$, then $ad-bc=1$ and $30\mid c$ imply that $a$ is
odd and coprime to $5$.

\begin{lemma}\label{lem.transform}
Let $\gamma=\smat abcd\in\Gamma_{0}(30)$ and $j\in\{1,2,3\}$.  Then
\begin{equation}\label{eq.gamact}
r_{j}(\gamma\tau)=
\begin{cases}
\zeta^{jab}\,r_{j}(\tau), & a\equiv\pm1\pmod5,\\[3pt]
-\,\zeta^{jab}\big/r_{j}(\tau), & a\equiv\pm2\pmod5.
\end{cases}
\end{equation}
\end{lemma}

\begin{proof}
If $c=0$, then $ad=1$, so $a=d=\pm1$.  Hence
\begin{equation*}
\gamma\tau=\frac{a\tau+b}{a}=\tau+ab.
\end{equation*}
Applying \eqref{eq.translation} to the
numerator and denominator in \eqref{eq.rgeneta} gives
\begin{equation*}
\begin{aligned}
r_{j}(\gamma\tau)
&=r_{j}(\tau+ab)\\
&=e^{\pi i\cdot5jab
\left(B_{2}(\frac15)-B_{2}(\frac25)\right)}r_{j}(\tau)\\
&=e^{2\pi ijab/5}r_{j}(\tau)
=\zeta^{jab}r_{j}(\tau),
\end{aligned}
\end{equation*}
which proves \eqref{eq.gamact} when $c=0$.

Now suppose that $c\ne0$. 
Since $5j\mid30\mid c$,  we may apply \eqref{eq.cne0}, with  $N=5j$, to the
numerator and denominator in \eqref{eq.rgeneta}.  The
$\varepsilon$-factors are the same and cancel.  Therefore

\begin{equation}\label{eq.transformraw}
\begin{aligned}
r_{j}(\gamma\tau)
&=
e^{\pi i\left(
\frac{(j^{2}-4j^{2})ab}{5j}-(j-2j)b
\right)}
\frac{\eta_{5j,aj}(\tau)}
{\eta_{5j,2aj}(\tau)}\\
&=
e^{\pi i\,jb(5-3a)/5}
\frac{\eta_{5j,aj}(\tau)}
{\eta_{5j,2aj}(\tau)}.
\end{aligned}
\end{equation}

Because $a$ is odd, $u:=\tfrac12(5-3a)\in\mathbb{Z}$.  As
$2u\equiv2a\pmod5$, we have $u\equiv a\pmod5$, and thus
\begin{equation}
e^{\pi i\,\frac{jb(5-3a)}{5}}=e^{2\pi i\,jbu/5}=\zeta^{jbu}=\zeta^{jab}.
\label{eq.expuniform}
\end{equation}

Since $a$ is odd and coprime to $5$, we may write
\begin{equation*}
a=10t+a_0,
\qquad
a_0\in\{1,9,7,3\},
\end{equation*}
where these four choices correspond, respectively, to
$a\equiv1,-1,2,-2\pmod5$. We then have
\begin{equation*}
aj=a_0j+2t(5j),
\qquad
2aj=2a_0j+4t(5j).
\end{equation*}
Thus, by \eqref{eq.indexshift},
\begin{equation*}
\frac{\eta_{5j,aj}(\tau)}
{\eta_{5j,2aj}(\tau)}
=
\frac{\eta_{5j,a_0j}(\tau)}
{\eta_{5j,2a_0j}(\tau)}.
\end{equation*}	
Using \eqref{eq.Nshift}, \eqref{eq.indexshift} and \eqref{eq.reflect}, we obtain
\begin{equation*}
\begin{array}{c|c|c|c}
a_{0}
& a\pmod5
& \eta_{5j,a_{0}j}(\tau)
& \eta_{5j,2a_{0}j}(\tau)\\ \hline
1 & 1
& \eta_{5j,j}(\tau)
& \eta_{5j,2j}(\tau)\\
9 & -1
& -\eta_{5j,j}(\tau)
& -\eta_{5j,2j}(\tau)\\
7 & 2
& -\eta_{5j,2j}(\tau)
& \eta_{5j,j}(\tau)\\
3 & -2
& \eta_{5j,2j}(\tau)
& -\eta_{5j,j}(\tau).
\end{array}
\end{equation*}
Therefore,
\begin{equation}\label{eq.quot}
\frac{\eta_{5j,aj}(\tau)}
{\eta_{5j,2aj}(\tau)}
=
\begin{cases}
r_{j}(\tau),
& a\equiv\pm1\pmod5,\\[3pt]
-r_{j}(\tau)^{-1},
& a\equiv\pm2\pmod5.
\end{cases}
\end{equation}
Substituting \eqref{eq.expuniform} and \eqref{eq.quot} into
\eqref{eq.transformraw} yields \eqref{eq.gamact}.
\end{proof}

\begin{lemma}\label{lem.modularity}
For all integers $m,n,p$, the function $G=G(m,n,p)$ of \eqref{eq.G-def} is
holomorphic on $\mathbb{H}$ and is a modular function on $\Gamma_{0}(30)$.
\end{lemma}

\begin{proof}
Let $\gamma=\smat abcd\in\Gamma_{0}(30)$.  As noted  {above},
we have $\gcd(a,5)=1$, so $a\equiv\pm1$ or $a\equiv\pm2\pmod5$.

If $a\equiv\pm1\pmod5$, then \eqref{eq.w-def} and Lemma~\ref{lem.transform} give
\begin{equation*}
\begin{aligned}
w(\gamma\tau)
&=\zeta^{ab\left[(-2m-n)+2(m-n-p)+3(n-p)\right]}w(\tau)\\
&=\zeta^{-5abp}w(\tau)
=w(\tau).
\end{aligned}
\end{equation*}
Similarly, if $a\equiv\pm2\pmod5$, then \eqref{eq.w-def} and
Lemma~\ref{lem.transform} yield
\begin{equation*}
\begin{aligned}
w(\gamma\tau)
&=(-1)^{(-2m-n)+(m-n-p)+(n-p)}
\zeta^{-5abp}w(\tau)^{-1}\\
&=(-1)^{m+n}w(\tau)^{-1}.
\end{aligned}
\end{equation*}
In either case, \eqref{eq.G-def} gives $G(\gamma\tau)=G(\tau)$.  Thus $G$ is
invariant under $\Gamma_{0}(30)$.

Each $\eta_{5j,g}$ is holomorphic and nonvanishing on
$\mathbb{H}$, so it follows from \eqref{eq.rgeneta} that
$r_{j}^{\pm1}$ is holomorphic on $\mathbb{H}$.  Hence, by \eqref{eq.w-def}
and \eqref{eq.G-def}, $G$ is holomorphic
on $\mathbb{H}$.

Finally, fix $\gamma\in\mathrm{SL}_{2}(\mathbb{Z})$.
Lemma~\ref{lem.yangorder} shows that each
$\eta_{5j,g}(\gamma\tau)$ has a $q$-expansion with finitely many negative exponents. It follows from
\eqref{eq.rgeneta}, \eqref{eq.w-def}, and \eqref{eq.G-def}
that the same is true of $r_j(\gamma\tau)$ and
$G(\gamma\tau)$. This completes
the proof.	
\end{proof}

{We now proceed to proofs of the initial conditions found in \eqref{eq.010} and \eqref{eq.001}.} 

\begin{proof}[Proof of \eqref{eq.010} and \eqref{eq.001}]
We need to show that
the two functions
\begin{align}
F_{1}&:=G(0,1,0)-\bigl(1-V\bigr),
\label{eq.F1}\\
F_{2}&:=G(0,0,1)-\bigl(U-2\bigr)
\label{eq.F2}
\end{align}
vanish identically. By Lemmas~\ref{lem.modularity} and~\ref{lem.newman}, the functions
$G(0,1,0)$, $G(0,0,1)$, $U$ and $V$, and hence $F_{1}$ and
$F_{2}$, are holomorphic on $\mathbb{H}$ and are modular functions on
$\Gamma_{0}(30)$. 

Lemma~\ref{lem.cusps} shows that $\Gamma_{0}(30)$ has eight
cusps, one for each divisor of $30$, represented by the points
$1/c$ with $c\mid30$. Here the cusps $1/1$ and $1/30$ are $\Gamma_{0}(30)$-equivalent to $0/1$ and
$\infty$, respectively. Lemma~\ref{lem.width} gives the widths of these cusps:
\begin{equation}
(h_{1/1},h_{1/2},h_{1/3},h_{1/5},h_{1/6},h_{1/10},h_{1/15},h_{1/30})
=(30,15,10,6,5,3,2,1).
\label{eq.widths}
\end{equation}

Fix $c\mid30$ and $j\in\{1,2,3\}$, and set $\ell:=\gcd(c,5j)$.
Applying Lemma~\ref{lem.yangorder} to the numerator and denominator of
\eqref{eq.rgeneta} gives
\begin{equation}
\ord\bigl(r_{j},\tfrac1c\bigr)
=\ord\bigl(\eta_{5j,j},\tfrac1c\bigr)-\ord\bigl(\eta_{5j,2j},\tfrac1c\bigr)
=\frac{\ell^{2}}{10j}
\Bigl(P_{2}\bigl(\tfrac{j}{\ell}\bigr)-P_{2}\bigl(\tfrac{2j}{\ell}\bigr)\Bigr).
\label{eq.nu-master}
\end{equation}
By \eqref{eq.w-def} and \eqref{eq.G-def},
\begin{equation*}
G(0,1,0)=-r_{1}r_{2}r_{3}^{-1}+\bigl(r_{1}r_{2}r_{3}^{-1}\bigr)^{-1},
\qquad
G(0,0,1)=r_{2}r_{3}+(r_{2}r_{3})^{-1}.
\end{equation*}
It follows that
\begin{equation*}
\ord\bigl(G(0,1,0),\tfrac1c\bigr)
\geq-\left|\ord\bigl(r_1r_2r_3^{-1},\tfrac1c\bigr)\right|,
\qquad
\ord\bigl(G(0,0,1),\tfrac1c\bigr)
\geq-\left|\ord\bigl(r_2r_3,\tfrac1c\bigr)\right|.
\end{equation*}
The  invariant orders for $r_{1}r_{2}r_{3}^{-1}$ and $r_{2}r_{3}$ obtained from \eqref{eq.nu-master} and the resulting bounds for $G(0,1,0)$ and $G(0,0,1)$ at the cusps other than $\infty$ are listed in Table~\ref{tab.rorders}.	

\begin{table}[H]
\centering
\renewcommand{\arraystretch}{1.25}
\begin{tabular}{|c||c|c|c|c|c|c|c|}
\hline
cusp $\tfrac1c$ & $\tfrac11$ & $\tfrac12$ & $\tfrac13$ & $\tfrac15$
& $\tfrac16$ & $\tfrac1{10}$ & $\tfrac1{15}$\\
\hline
$\ord\bigl(r_1r_2r_3^{-1},\tfrac1c\bigr)$ & $0$ & $0$ & $0$
& $\tfrac16$ & $0$ & $\tfrac23$ & $-\tfrac12$\\
$\ord\bigl(r_2r_3,\tfrac1c\bigr)$ & $0$ & $0$ & $0$
& $-\tfrac16$ & $0$ & $\tfrac13$ & $\tfrac12$\\
\hline
$\ord\bigl(G(0,1,0),\tfrac1c\bigr)\ge$ & $0$ & $0$ & $0$
& $-\tfrac16$ & $0$ & $-\tfrac23$ & $-\tfrac12$\\
$\ord\bigl(G(0,0,1),\tfrac1c\bigr)\ge$ & $0$ & $0$ & $0$
& $-\tfrac16$ & $0$ & $-\tfrac13$ & $-\tfrac12$\\
\hline
\end{tabular}
\caption{Invariant orders and lower bounds at the cusps of $\Gamma_0(30)$ other than $\infty$}
\label{tab.rorders}
\end{table}

Since $G(m,n,p)$ is invariant under $\Gamma_{0}(30)$, we get
\begin{equation*}
\Ord_{\Gamma_{0}(30)}\bigl(G(m,n,p),\tfrac1c\bigr)
=h_{1/c}\,\ord\bigl(G(m,n,p),\tfrac1c\bigr)
\end{equation*}
from \eqref{eq.twoorders}. Multiplying the last two rows of Table~\ref{tab.rorders} by the
corresponding widths in \eqref{eq.widths} produces the first two rows of
Table~\ref{tab.Gorders}. The next two rows contain the orders of $U$
and $V$ from Lemma~\ref{lem.ligorder}.
Finally, \eqref{eq.F1} and
\eqref{eq.F2} give
\begin{equation*}
\Ord_{\Gamma_{0}(30)}
\bigl(F_{1},\tfrac1c\bigr)
\geq
\min\!\left\{
\Ord_{\Gamma_{0}(30)}
\bigl(G(0,1,0),\tfrac1c\bigr),
\Ord_{\Gamma_{0}(30)}
\bigl(V,\tfrac1c\bigr),
0
\right\},
\end{equation*}
and
\begin{equation*}
\Ord_{\Gamma_{0}(30)}
\bigl(F_{2},\tfrac1c\bigr)
\geq
\min\!\left\{
\Ord_{\Gamma_{0}(30)}
\bigl(G(0,0,1),\tfrac1c\bigr),
\Ord_{\Gamma_{0}(30)}
\bigl(U,\tfrac1c\bigr),
0
\right\}.
\end{equation*}
These minima yield the last two rows of Table~\ref{tab.Gorders}.

\begin{table}[H]
\centering
\renewcommand{\arraystretch}{1.25}
\begin{tabular}{|c||c|c|c|c|c|c|c|}
\hline
cusp $\tfrac1c$ & $\tfrac11$ & $\tfrac12$ & $\tfrac13$ & $\tfrac15$
& $\tfrac16$ & $\tfrac1{10}$ & $\tfrac1{15}$\\
\hline
$\Ord_{\Gamma_{0}(30)}\bigl(G(0,1,0),\tfrac1c\bigr)\ge$
& $0$ & $0$ & $0$ & $-1$ & $0$ & $-2$ & $-1$\\
$\Ord_{\Gamma_{0}(30)}\bigl(G(0,0,1),\tfrac1c\bigr)\ge$
& $0$ & $0$ & $0$ & $-1$ & $0$ & $-1$ & $-1$\\
\hline
$\Ord_{\Gamma_{0}(30)}\bigl(U,\tfrac1c\bigr)$ & $0$ & $2$ & $2$ & $-1$ & $0$ & $-1$ & $-1$\\
$\Ord_{\Gamma_{0}(30)}\bigl(V,\tfrac1c\bigr)$ & $2$ & $1$ & $0$ & $-1$ & $1$ & $-2$ & $-1$\\
\hline
$\Ord_{\Gamma_{0}(30)}\bigl(F_{1},\tfrac1c\bigr)\ge$ & $0$ & $0$ & $0$ & $-1$ & $0$ & $-2$ & $-1$\\
$\Ord_{\Gamma_{0}(30)}\bigl(F_{2},\tfrac1c\bigr)\ge$ & $0$ & $0$ & $0$ & $-1$ & $0$ & $-1$ & $-1$\\
\hline
\end{tabular}
\caption{Orders and lower bounds at the cusps of $\Gamma_0(30)$ other than $\infty$}
\label{tab.Gorders}
\end{table}

Away from infinity, the bounds for $F_{1}$ and $F_{2}$ sum to $-4$
and $-3$, respectively.  
Direct $q$-expansions at $\infty$ give
\begin{align*}
G(0,1,0)&=2q+q^{2}-3q^{3}+q^{4}+O(q^{5}),\\
1-V     &=2q+q^{2}-3q^{3}+q^{4}+O(q^{5}),\\
G(0,0,1)&=q^{-1}+2q+q^{2}-q^{3}+O(q^{4}),\\
U-2     &=q^{-1}+2q+q^{2}-q^{3}+O(q^{4}),
\end{align*}
so that, by \eqref{eq.F1} and \eqref{eq.F2},
\begin{equation*}
\Ord_{\Gamma_{0}(30)}(F_{1},\infty)\geq5,
\qquad
\Ord_{\Gamma_{0}(30)}(F_{2},\infty)\geq4.
\end{equation*}
Lemma~\ref{lem.valence} therefore implies
$F_{1}=F_{2}=0$, which proves \eqref{eq.010} and \eqref{eq.001}.	

\end{proof}

\section{Proofs of \eqref{eq.recm}, \eqref{eq.recn}, and \eqref{eq.recp}}
\label{sec.product}
{We begin this section by proving the following important lemma.}
\begin{lemma}
For all integers $m,n,p,i,j,k$, we have
\begin{equation}
G(m,n,p)\,G(i,j,k)
=G(m+i,\,n+j,\,p+k)
+(-1)^{i+j}G(m-i,\,n-j,\,p-k).
\label{eq.productG}
\end{equation}
\end{lemma}

\begin{proof}
By \eqref{eq.w-def} and \eqref{eq.G-def},
\begin{align*}
&G(m,n,p)\,G(i,j,k)\\
&=\bigl(r_{1}^{-2m-n}r_{2}^{m-n-p}r_{3}^{n-p}
+(-1)^{m+n}r_{1}^{2m+n}r_{2}^{-m+n+p}r_{3}^{-n+p}\bigr)\\
&\qquad\times\bigl(r_{1}^{-2i-j}r_{2}^{i-j-k}r_{3}^{j-k}
+(-1)^{i+j}r_{1}^{2i+j}r_{2}^{-i+j+k}r_{3}^{-j+k}\bigr)\\
&=\Bigl(
r_{1}^{-2(m+i)-(n+j)}
r_{2}^{(m+i)-(n+j)-(p+k)}
r_{3}^{(n+j)-(p+k)}\\
&\qquad\quad
+(-1)^{(m+i)+(n+j)}
r_{1}^{2(m+i)+(n+j)}
r_{2}^{-(m+i)+(n+j)+(p+k)}
r_{3}^{-(n+j)+(p+k)}\Bigr)\\
&\quad+(-1)^{i+j}\Bigl(
r_{1}^{-2(m-i)-(n-j)}
r_{2}^{(m-i)-(n-j)-(p-k)}
r_{3}^{(n-j)-(p-k)}\\
&\qquad\quad
+(-1)^{(m-i)+(n-j)}
r_{1}^{2(m-i)+(n-j)}
r_{2}^{-(m-i)+(n-j)+(p-k)}
r_{3}^{-(n-j)+(p-k)}\Bigr)\\
&=G(m+i,n+j,p+k)
+(-1)^{i+j}G(m-i,n-j,p-k).
\end{align*}
\end{proof}

\begin{proof}[Proof of {\eqref{eq.recm}, \eqref{eq.recn}, and \eqref{eq.recp}}]

Taking $(i,j,k)$ to be $(1,0,0)$, $(0,1,0)$ and $(0,0,1)$ in
\eqref{eq.productG} and rearranging, we obtain
\begin{align*}
G(m+1,n,p)&=G(1,0,0)\,G(m,n,p)+G(m-1,n,p),\\
G(m,n+1,p)&=G(0,1,0)\,G(m,n,p)+G(m,n-1,p),\\
G(m,n,p+1)&=G(0,0,1)\,G(m,n,p)-G(m,n,p-1).
\end{align*}
Substituting the initial values \eqref{eq.100}, \eqref{eq.010}, and
\eqref{eq.001} yields \eqref{eq.recm}, \eqref{eq.recn}, and
\eqref{eq.recp}, respectively.
\end{proof}

\bibliographystyle{abbrv}
%\nocite{*}
\bibliography{rogers_ramanujan}

\end{document}